\documentclass[12pt, reqno]{amsart}

\usepackage{amsthm,amssymb,amsmath,amscd}
\usepackage{amsthm}
\usepackage{graphicx,subcaption}
\usepackage{verbatim}
\usepackage{todonotes}
\usepackage[backref]{hyperref}
\hypersetup{
	colorlinks,
	linkcolor={red!60!black},
	citecolor={green!60!black},
	urlcolor={blue!60!black}
}

\usepackage{geometry}
\usepackage{enumitem}
\def\rmlabel{\upshape({\itshape \roman*\,})}

\usepackage{amsmath,amssymb,amsthm}
\usepackage{amsrefs}

\usetikzlibrary{decorations.pathreplacing,calc}

\usepackage[backref]{hyperref}

\theoremstyle{definition}
\newtheorem{definition}{Definition}

\theoremstyle{plain}
\newtheorem{theorem}[definition]{Theorem}
{Lemma}
{Theorem}
\newtheorem{lemma}[definition]{Lemma}

{Conjecture}

\theoremstyle{remark}

\newcommand{\ex}{\mathrm{ex}}

\makeatletter
\newsavebox\myboxA
\newsavebox\myboxB
\newlength\mylenA

\newcommand*\xoverline[2][0.75]{%
	\sbox{\myboxA}{$\m@th#2$}%
	\setbox\myboxB\null
	\ht\myboxB=\ht\myboxA%
	\dp\myboxB=\dp\myboxA%
	\wd\myboxB=#1\wd\myboxA
	\sbox\myboxB{$\m@th\overline{\copy\myboxB}$}
	\setlength\mylenA{\the\wd\myboxA}
	\addtolength\mylenA{-\the\wd\myboxB}%
	\ifdim\wd\myboxB<\wd\myboxA%
	\rlap{\hskip 0.5\mylenA\usebox\myboxB}{\usebox\myboxA}%
	\else
	\hskip -0.5\mylenA\rlap{\usebox\myboxA}{\hskip 0.5\mylenA\usebox\myboxB}%
	\fi}
\makeatother

\begin{document}

\title[Ramsey numbers of even cycles versus stars ]
{The Ramsey number of \\a long even cycle versus a star }

\author{Peter Allen}
\address{Department of Mathematics\\
	London School of Economics\\
	 London, United Kingdom}
\email{\tt p.d.allen@lse.ac.uk}
\author{Tomasz \L{}uczak}
\address{Adam Mickiewicz University\\
Faculty of Mathematics and Computer Science\\
Pozna\'n, Poland}
\email{\tt tomasz@amu.edu.pl}
\author{Joanna Polcyn}
\address{Adam Mickiewicz University\\
Faculty of Mathematics and Computer Science\\
Pozna\'n, Poland}
\email{\tt joaska@amu.edu.pl}
\author{Yanbo Zhang}
\address{Hebei Normal University\\
School of Mathematical Sciences\\ Shijiazhuang, P.R.China}
\email{\tt ybzhang@163.com}
\thanks{The second author was partially supported by National Science Centre, Poland, grant 2017/27/B/ST1/00873.
	The fourth author was partially supported by NSFC under grant numbers 11601527, 11971011, and 11801520.}

\begin{abstract}
We find the exact value of the Ramsey number $R(C_{2\ell},K_{1,n})$, when $\ell$ and   $n=O(\ell^{10/9})$
are large. Our result is closely related to the behaviour of
 Tur\'an number $\ex(N, C_{2\ell})$ for an even cycle 
whose length grows quickly with $N$.
\end{abstract}



\maketitle

\date{October 16th, 2020}

\section{Introduction}

For a graph $H$ by
$$\ex(N,H)=\max\{|E|: G=(V,E)\not\supseteq H  \textrm{\ and\ } |V|=N\}$$
we denote its Tur\'an number. 
Let us recall that for graphs $H$ with chromatic number at least three the asymptotic value of $\ex(N,H) $ was
determined over fifty years ago by Erd\H{o}s and Stone~\cite{ESt}, and Erd\H{o}s and Simonovits~\cite{ESi}, while for most bipartite 
graphs $H$  the behaviour of $\ex(N, H)$ is not well understood. Let us recall some results on the case  when $H$ is an even cycle $C_{2\ell}$. 
The best upper bound for  $\ex(N, C_{2\ell})$ for general $\ell$ is due to Bukh and Jiang~\cite{BJ}, who improved the classical theorem of Bondy and Simonovits~\cite{BS} to 
$$\ex(N, C_{2\ell})\le 80\sqrt{\ell} \ln \ell N^{1+1/\ell}+ 10\ell^2 N. $$
The best lower bound which holds for all $\ell$ follows from the construction of  regular graphs 
of large girth  by   Lubotzky,  Phillips, and Sarnak~\cite{LPS}, which gives 
$$\ex(N, C_{2\ell})\ge N^{1+(2+o(1))/3 \ell}.$$
The correct exponent $\alpha_\ell$ for which $\ex(N,C_{2\ell})=N^{\alpha_\ell+o(1)}$ is known only for $\ell=2,3,5$, when it is equal  to 
$1+1/\ell$   (see the survey of F\"uredi and Simonovits~\cite{FS}  and references therein), and finding it for 
every $\ell$ is one of the major open problems in extremal graph theory. 
Can it become easier when we allow
the length of an even cycle to grow with $N$? This paper was inspired by this question.  However,   
instead of the original problem we  consider its,  nearly  equivalent,
partition version, and, instead of $\ex(N,C_{2\ell})$, we study the Ramsey number   
$R(C_{2\ell}, K_{1,n})$. Although typically we define Ramsey numbers using colours, here and below 
we rather view $R(C_{2\ell}, K_{1,n})$ as the minimum $N$ such that each graph on $N$ vertices and 
minimum degree at least $N-n$ contains a copy of $C_{2\ell}$. 
Note that if $\ex(N,H)\le M$, then 
$$R(H, K_{1,N-\lfloor 2M/N\rfloor})\le N\,.$$
On the other hand, if there exists a $d$-regular $H$-free graph on $N$ vertices, then, clearly, 
$$R(H, K_{1,N-d}) >N\,.$$
Consequently, from the result of Bukh and Jiang and the construction of Lubotzky, Phillips, and Sarnak mentioned above we get  
\begin{equation}\label{eq1}
n+ 2n^{(2+o(1))/3\ell} \le R(C_{2\ell}, K_{1,n})\le n+ 161\sqrt{\ell} \ln \ell n^{1/\ell}+ 22\ell^2.
\end{equation} 
Since a graph on $N$ vertices with  minimum degree at least $N/2$ is hamiltonian (Dirac~\cite{D}), 
and if its  minimum degree is larger than $N/2$, it is 
pancyclic (Bondy~\cite{B}), for  $\ell \ge n\ge 2 $, we have $R(C_{2\ell},K_{1,n})=2\ell$. 
Moreover, Zhang, Broersma, and Chen~\cite{ZBC} showed that 
if $n/2<\ell<n$ then $R(C_{2\ell},K_{1,n})=2n$, while for  $3n/8+1\le \ell\le n/2$, we get $R(C_{2\ell},K_{1,n})=4\ell-1$. Our main result 
determines the value of  $R(C_{2\ell},K_{1,n})$ for all large $\ell$, and  $n$ not much larger than $\ell$. 

\begin{theorem}\label{thm1}
For every $t\ge 2$,  $\ell\ge (19.1t)^9$, and $n$ such that   $(t-1)(2\ell-1)\le n-1 < t(2\ell-1)$, we have 
\begin{equation*}\label{eq:main}
	R(C_{2\ell}, K_{1,n}) =f_t(\ell,n)+1,
\end{equation*}
where 
\begin{equation*}\label{eq:main2}
f_t(\ell,n)=\max\{t(2\ell-1), n+ \lfloor (n-1)/t\rfloor  \}.
\end{equation*}

\end{theorem}

The condition $\ell\ge (19.1t)^9$ in Theorem~\ref{thm1} above, which holds when, say, $n\le 0.1\ell^{10/9}$, follows from the bounds given by Lemma~\ref{l:GHS} below, one of the key ingredients of 
 our argument. It is certainly far from being optimal and we  suspect that the result  
holds for any $n$ growing polynomially with $\ell$, but it is  conceivable that it remains true  even for $n$ which grows exponentially with $\ell$. On the other hand, because of \eqref{eq1}, the assertion of Theorem~\ref{thm1} fails for, say,  $n\ge \ell^{2\ell}$. 

We remark that one can use a similar technique to find the value of $\ex(N,C_{2\ell})$ when $N$ is not much larger than 
$\ell$. In this case $\ex(N,C_{2\ell})$ is the same as $\ex(N,C_{\ge 2\ell})$, i.e. the maximum number of edges in a graph on $N$ vertices which contains no cycles of length {\sl at least} $2\ell$. Let us recall that $\ex(N,C_{\ge 2\ell})$,  was determined  already by Erd\H{o}s and Gallai~\cite{EG} who also described the structure of all extremal graphs for this problem -- it turns out that 
all their blocks, except perhaps one,  are cliques on $2\ell-1$ vertices. However,      
the behaviour of $R(C_{2\ell},K_{1,n})$ 
seemed to us more intriguing. Indeed, for a given $\ell$ and $(t-1)(2\ell-1)\le n-1< \frac{t^2}{t+1}(2\ell-1)$ we have 
$$  f_t(\ell,n) =(2\ell-1)t,$$
i.e. for this range of $n$ the value of  $R(C_{2\ell}, K_{1,n})$ does not depend on the size of the star. On the other hand, as is shown in the next section, for $\frac{t^2}{t+1}(2\ell-1)\le n-1< t(2\ell-1) $, when
$$ f_t(\ell,n) = n+ \lfloor (n-1)/t\rfloor\,$$
the `extremal graphs' which determine the value of $R(C_{2\ell}, K_{1,n})$  typically have all blocks smaller than $2\ell-1$.

\section{The lower bound for $R(C_{2\ell}, K_{1,n})$}\label{sec:lower}

In this section we show that for given integers $t$, $\ell$, and $n$ such that $(t-1)(2\ell-1) \le n-1 < t(2\ell - 1)$, we have
\begin{equation}\label{eq:lower}
R(C_{2\ell},K_{1,n}) >f_t(\ell,n)=\max\{t(2\ell-1), n+ \lfloor (n-1)/t\rfloor \}.
\end{equation}

Let us consider first the graph $H_1$ which consists of $t$ vertex-disjoint copies of the complete graph $K_{2\ell-1}$. 
Clearly, $|V(H_1)| = t(2\ell-1)$ and $H_1\nsupseteq C_{2\ell}$. 
Moreover, $\Delta(\xoverline H_1) = (t-1)(2\ell-1) \le n-1$ yielding $\xoverline H_1 \nsupseteq K_{1,n}$. Hence
$$R(C_{2\ell},K_{1,n}) >t(2\ell-1)\,.$$

Now let  $k=n-1 - t\lfloor (n-1)/t\rfloor$ and $m=\lfloor (n-1)/t\rfloor + 1$. We define a graph $H_2$ 
as a union of $k$ vertex-disjoint complete graphs $K_m$ and $t+1-k$ other copies of $K_m$ which are `almost'
vertex-disjoint except that  they share exactly one vertex (see Figure~\ref{fig}). 

	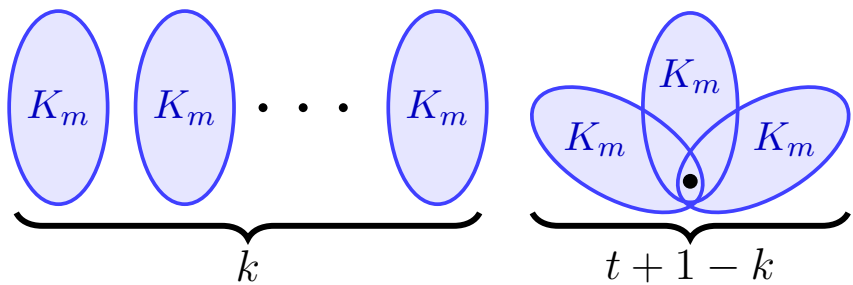
\begin{figure}[h!]
	
	\scalebox{1.4}{%
	\begin{tikzpicture}		
		\def\size{(.45cm  and 0.9cm)}
		\def \cin{blue!10!white}
		\def \cout{blue!75!white}
		
		\draw[\cout, line width=2pt,fill opacity=1] (-6,0) ellipse \size;
		\fill[\cin,opacity=1] (-6,0) ellipse \size;
		\node at (-6,0) {\large\textcolor{blue!75!black}{\small $K_m$}};
		
		\draw[\cout, line width=2pt,fill opacity=1] (-4.8,0) ellipse \size;
		\fill[\cin,opacity=1] (-4.8,0) ellipse \size;
		\node at (-4.8,0) {\large\textcolor{blue!75!black}{\small $K_m$}};
		
		\node at (-3.6,0) {\Huge $\dots$};
		
		\draw[\cout, line width=2pt,fill opacity=1] (-2.4,0) ellipse \size;
		\fill[\cin,opacity=1] (-2.4,0) ellipse \size;
		\node at (-2.4,0) {\large\textcolor{blue!75!black}{\small $K_m$}};

		\draw[ultra thick, decorate,decoration={brace,amplitude=7pt,mirror}] 
		($(-6.4,-1)$) coordinate (a)  -- ($(-2,-1)$) coordinate (c) ; 
		
		\node at (-4.2,-1.5){ $k$};

		\fill (0,0) circle (3pt);
		
		\fill[\cin,opacity=1, rotate around={0:(0,-.8)}] (0,0) ellipse \size;
		\fill[\cin,opacity=1,rotate around={-60:(0,-.8)}] (0,0) ellipse \size;
		\fill[\cin,opacity=1, rotate around={60:(0,-.8)}] (0,0) ellipse \size;
		
		\draw[\cout, line width=1pt,fill opacity=1, rotate around={60:(0,-.8)}] (0,0) ellipse \size;
		\draw[\cout, line width=1pt,fill opacity=1,rotate around={-60:(0,-.8)}] (0,0) ellipse \size;
		\draw[\cout, line width=1pt,fill opacity=1, rotate around={0:(0,-.8)}] (0,0) ellipse \size;
		
		\node at (0,.3) {\large\textcolor{blue!75!black}{\small $K_m$}};
		\node at (-.9,-.3) {\large\textcolor{blue!75!black}{\small $K_m$}};
		\node at (.9,-.3) {\large\textcolor{blue!75!black}{\small $K_m$}};

		\fill (0,-.7) circle (2pt);		
		
		\draw[ultra thick, decorate,decoration={brace,amplitude=7pt,mirror}] 
		($(-1.5,-1)$) coordinate (a)  -- ($(1.5,-1)$) coordinate (c) ; 
		
		\node at (0,-1.5){$t+1-k$};

	\end{tikzpicture}
	}
	\caption {Graph $H_2$. Here $k=n-1-t\lfloor\frac{n-1}{t}\rfloor$ and $m=\left\lfloor\frac{n-1}{t}\right\rfloor+1$}
	\label{fig}
\end{figure}

 Then 
\begin{align*}
|V(H_2)|&=km+(t+1-k)(m-1) + 1=(t+1)m-(t-k)\\
&=(t+1)(\lfloor (n-1)/t\rfloor + 1) -t + n-1 - t\lfloor (n-1)/t\rfloor\\
&=n+\lfloor (n-1)/t\rfloor.			
\end{align*}
Note also that  $n-1< t(2\ell-1)$, and so 
$m=\lfloor (n-1)/t\rfloor +1\le 2\ell-1$.  Hence $H_2\not\supseteq C_{2\ell}$.
Finally, 
\begin{align*}
\Delta(\xoverline H_2) = |V| - m= n+\lfloor (n-1)/t\rfloor   -\lfloor (n-1)/t\rfloor - 1 =n-1.
\end{align*}
Therefore 
$$R(C_{2\ell},K_{1,n}) >|V(H_2)|= n+ \lfloor (n-1)/t\rfloor,$$
and \eqref{eq:lower} follows.

Let us remark that the two graphs $H_1$ and $H_2$ we used above are by no means the only `extremal graphs' with $R(C_{2\ell},K_{1,n})-1$ vertices. Let us take,  for example, $n=4.1\ell$. Then 
$R(C_{2\ell},K_{1,n})=3(2\ell-1)+1$ and the lower bound  for $R(C_{2\ell},K_{1,n})$ is `certified' by  the graph $H'_1$ which consists of 
three vertex disjoint cliques $K_{2\ell-1}$. However, if we replace each of these cliques by a graph on $2\ell-1$ vertices and minimum degree $1.91\ell$, the complement of the resulting graph will again contain 
no $K_{1,n}$, so each such graph shows that $R(C_{2\ell},K_{1,n})>3(2\ell-1)$ as well. On the other hand,  adding to $H'_1$ a triangle with vertices in different cliques does not result in a copy of $C_{2\ell}$, so $H'_1$ is not even a maximal extremal graph certifying that   $R(C_{2\ell},K_{1,n})>3(2\ell-1)$.

\section{Cycles in 2-connected graphs}

In order to show the upper bound for $R(C_{2\ell},K_{1,n})$ we have to argue that large graphs with a sufficiently large minimum degree contain $C_{2\ell}$. In  this section we collect a number of  results on cycles in 2-connected graphs we shall use later on.

 Let us recall first that the celebrated  theorem of Dirac~\cite{D} states that each 2-connected graph $G$ on $n$ vertices contains a cycle of length at least $\min\{ 2\delta(G),n\}$, and, in particular, each graph with minimum degree at least $n/2$ is hamiltonian. 
 Below we mention some generalizations of this result. 
Since we are interested mainly in even cycles, we start with the following observation due to Voss and Zuluaga \cite{VZ}.

\begin{lemma}
\label{l:VZ}
Every $2$-connected graph $G$ on $n$ vertices contains an even cycle $C$ of length at least $\min\{2\delta(G), n-1\}$.\qed
\end{lemma}	

The following  result by 
Bondy and Chv\'{a}tal~\cite{BC} shows that the condition $\delta(G)\ge n/2$, sufficient for hamiltonicity, can be replaced by a somewhat weaker one. 
 Recall that the closure of a graph $G=(V,E)$ is the graph obtained from $G$ by recursively joining pairs of non-adjacent vertices whose degree sum is at least $|V|$ until no such pair remains.

\begin{lemma}
\label{l:BC}
A graph  $G$ is hamiltonian if and only if its closure is hamiltonian. \qed  	
\end{lemma}  

If we allow $\delta(G)>n/2$, then, as observed by Bondy~\cite{B}, $G$ becomes pancyclic. We  use the following strengthening
of this result, proved under slightly stronger assumptions, due to Williamson~\cite{W}. 

\begin{lemma}
\label{l:W}
Every graph $G=(V,E)$ on $n$ vertices  with $\delta(G)\ge  n/2+1$ has the following property. For every $v,w\in V$ and every $k$ such that $2\le k\le n-1$, $G$ contains 
a path of length $k$  which starts at $v$ and ends at $w$.  In particular,  $G$ is pancyclic. \qed
\end{lemma}

Finally, we state a theorem of  Gould, Haxell, and Scott~\cite{GHS}, which is crucial for our argument. 
Here and below $\textrm{ec}(G)$ denotes the length of the longest even cycle in $G$.

\begin{lemma}
\label{l:GHS}
Let $a>0$, $\hat K = 75\cdot 10^4 a^{-5}$, and $G$ be a graph with $n\ge 45 \hat K/a^4$ vertices and minimum degree at least $an$. Then
for every even  $r\in [4,\textrm{ec}(G)-\hat K]$,  $G$ contains a cycle of length $r$. 
\end{lemma}

Let us also note the following consequence of the above results. 

\begin{lemma}\label{l:small}
For  $c\ge 1$ we set
\begin{equation}\label{eq:ks}
K(c) = 24\cdot 10^6 c^5 = 75\cdot 10^4 (1/2c)^{-5},
\end{equation}	
and let $\ell \ge 360c^4K(c)$. Then for every $2$-connected $C_{2\ell}$-free graph $H=(V,E)$ such that $|V|\le 2\ell c$ and
 $\delta(H) \ge \ell + K(c)$, we have
	\[
	|V|\le 2\ell - 1.
	\]
\end{lemma}

\begin{proof}	
Let us consider first the case   $|V| < 2\ell + 2K(c)-2$. Then, since 
	\[
	\delta(H) \ge \ell +K(c) > { |V|}/{2}+1,
	\]
from Lemma~\ref{l:W} we infer  that $H$ is pancyclic. But $C_{2\ell} \nsubseteq H$ meaning that $|V|\le 2\ell - 1$, as required. 
	
	On the other hand, for $|V|\ge 2\ell + 2K(c)-2$  Lemma~\ref{l:VZ} implies that 
	\[
	\textrm{ec}(H) \ge 2\ell + 2K(c)-2>2\ell+K(c)
	\]
	Moreover, as $|V|\le 2\ell c$ and $\ell \ge 360c^4K(c)$, one gets 
	\[
	\delta(H)> \ell \ge \frac{1}{2c}|V|\quad \textrm{and}\quad |V| > 2\ell \ge 45\left(\frac{1}{2c}\right)^{-4}K(c).
	\]
	Therefore, from Lemma \ref{l:GHS} applied to $H$ with $a=1/(2c)$, we infer that $H$ contains a cycle of length ${2\ell}$, contradicting $C_{2\ell}$-freeness of $H$.
\end{proof}

\section{Proof of the main result}

The two examples of graphs we used to verify the lower bound for $R(C_{2\ell}, K_{1,n})$ (see Section~\ref{sec:lower}) 
suggest that a natural way to deal with the upper bound for $R(C_{2\ell}, K_{1,n})$ is to show first that each 
$C_{2\ell}$-free graph $G$ with a large minimum degree has all blocks smaller than $2\ell$.  However, most results on the existence of cycles in 2-connected graphs 
use the minimum degree condition, and even if the minimum degree of $G$ is large, some of its blocks may contain vertices of small degree.   
Nonetheless we shall prove that the set of vertices in each such $G$  contains a `block-like' family of  2-connected subgraphs without vertices of very small degree.  Then, based on the results of the last section, we argue that each subgraph  in such  family is small. 
In the third and final part of our proof we show that if this is the case, then $G$ 
has at most $f_t(\ell, n)$ vertices.  

Before the proof of Theorem~\ref{thm1} we  state two technical lemmata. The first one will become instrumental in the first part of our argument, 
when we decompose the graph $G$ into 2-connected subgraphs without vertices of small degree. 

\begin{lemma}\label{l:dec}
	Let  $n\ge k\ge 2$. 
For each graph $G$ with $n$ vertices and minimum degree $\delta(G)\ge n/k + k$, 	there exists an  $s<k$ and a set of vertices $U\subset V(G)$, $|U| \le s-1$, such that $G-U$ is a union of $s$ vertex-disjoint $2$-connected graphs. 
\end{lemma}

\begin{proof}
Consider a sequence  $U_0,U_1,\dots, U_{t}=U$ of subsets of $V$  which starts with $U_0=\emptyset$  and, if $G-U_i$ contains a cut vertex $v_i$,  we put 
$U_{i+1}=U_i\cup \{v_i\}$. The process terminates when each component of $G-U_i$ is 2-connected.
Note that in each step the number of components of a graph increases by at least one, so $G-U_i$ has at least $i+1=|U_i|+1$ components. Moreover, the process must terminate 
for $t< k-1$ since otherwise the graph $G-U_{k-1}$ would have $n-k+1$ vertices,  
at least $k$ components, and the minimum degree at least $n/k+1$ which, clearly, is impossible. 
Hence the graph  $G-U=G-U_t$ has $n-t$ vertices, $s\ge |U|+1= t+1$ components, and minimum degree larger than $n/k+1$. Finally, let us notice that, again, since each component has more than $n/k$ vertices, we must have $s<k$.    
\end{proof}

The following result is crucial for the final stage of our argument, when we show that each graph $G$ with a large minimum degree, 
which admits a certain block-like decomposition into small 2-connected subgraphs, cannot be too large. 

\begin{lemma}\label{l:V}
	For a given set $V$ and positive integers $\ell, s,t,n\ge 2$, satisfying $(t-1)(2\ell-1)\le n-1 < t(2\ell-1)$, 
	let $V_1, V_2, \dots, V_s$ be subsets of $V$ such that 
	\begin{enumerate}[label=\rmlabel]
		\item \label{it:1} $V = V_1\cup V_2\cup \dots \cup V_s$,
		\item \label{it:2} $|V_i|\le 2\ell-1$ for $i=1,2,\dots,s$,
		\item \label{it:3} $|V\setminus V_i|\le n-1$ for  $i=1,2,\dots,s$, 
		\item \label{it:4} $|V_1| + |V_2| + \cdots + |V_s| \le |V| + s -1$.
	\end{enumerate}
	Then 
	\[
	|V|\le f_t(\ell,n)= \max \{t(2\ell-1),n+\lfloor (n-1)/t\rfloor\}\,.
	\]
\end{lemma}
\begin{proof} Note first that if $s\le t$, then \ref{it:1} and \ref{it:2} imply that $|V|\le t(2\ell-1)$. Thus, let us assume that 
$s\ge t+1$. Then, 
	\begin{eqnarray*}
		s(n-1)
		&\overset{\textrm{\ref{it:3}}}{\ge}&
		\sum_{i=1}^s |V\setminus V_i|=s|V| - (|V_1|+|V_2|+\dots+|V_s|) \\
		&\overset{\textrm{\ref{it:4}}}{\ge }&
		s|V| - (|V|+s-1) = (s-1)|V| - (s-1),
	\end{eqnarray*}
	and thereby
	\[
		\hfill	|V| \le \frac{s}{s-1}(n-1) +1 = n + \frac{n-1}{s-1}\le n+ \frac{n-1}{t}.
	\]
	Since $|V|$ is an integer, the assertion follows. 
\end{proof}
  
\begin{proof}[Proof of Theorem~\ref{thm1}]
Since we have already bounded $R(C_{2\ell},K_{1,n})$ from below in Section~\ref{sec:lower}, 
we are left with the task of showing that 
\begin{equation*}\label{l:up}
R(C_{2\ell},K_{1,n}) \le f_t(\ell,n)+1.
\end{equation*}
For this purpose, let $t\ge 2$, 
\begin{equation*}\label{eq:el}
\ell \ge  (19.1t)^9> 
360(t+1)^4\cdot K(t+1),
\end{equation*}
where $K(t+1) = 24\cdot 10^6 (t+1)^5$ is the function defined in \eqref{eq:ks}, 
and 
\begin{equation*}
(t-1)(2\ell-1) \le n-1 < t(2\ell-1).
\end{equation*}
Moreover, let $G=(V,E)$ be a $C_{2\ell}$-free graph on 
\begin{equation*}\label{eq:Nf}
|V| 
=f_t(\ell,n)+1
\end{equation*}
vertices such that $\xoverline G \nsupseteq K_{1,n}$ (or equivalently, $\Delta(\xoverline G) \le n-1$).

Recall that $f_t(\ell, n) = \max\{t(2\ell-1), n+ \lfloor (n-1)/t\rfloor \}$ and observe that
	\begin{equation}\label{eq:estimate}
	(n-1)+\frac{ t(2\ell-1)}{t+1} <
	f_t(\ell,n)
	< (t+1)(2\ell-1)\,.
\end{equation}
Indeed, the upper bound follows immediately from the fact that $n-1 < t(2\ell - 1)$,
so it is enough to verify the lower bound for $f_t(\ell,n)$. If 
	\[
		(n-1) + \frac{t(2\ell-1)}{t+1} < {t(2\ell-1)}  
	\]
then we are done, otherwise we have   $$ \frac{t(2\ell-1)}{t+1}\le\frac{n-1}{t}$$ and,
since $f_t(\ell,n)\ge n+\lfloor \frac{n-1}{t}\rfloor$, \eqref{eq:estimate} holds as well.

Our aim is to show that $G$ contains a family of 2-connected subgraphs $G_i=(V_i,E_i)$, $i=1,2,\dots,s$, such that their vertex sets fulfil the conditions 
\ref{it:1}-\ref{it:4} listed in Lemma~\ref{l:V}.


We first  apply Lemma \ref{l:dec} to $G$ with $k = \frac{(t+1)^2+1}{t}$. We are allowed to do this, because \eqref{eq:estimate} tells us that 
\begin{equation}\label{eq:dd}
\begin{aligned}
\delta(G) &= |V|-1 - \Delta(\xoverline G) \ge f_t(\ell,n)-(n-1)>  \frac{t(2\ell-1)}{t+1}\ge \frac{t|V|}{(t+1)^2}
\end{aligned}
\end{equation}
However, both $|V|$ and $\ell$ are much larger than $t$, in particular, $|V|\ge 2\ell> (19.1t)^9$. Hence, 
\[
\delta(G) \ge \frac{t|V|}{(t+1)^2} 
>
\frac{t}{(t+1)^2+1}|V| + \frac{(t+1)^2+1}{t}	
\]
and the assumptions of Lemma \ref{l:dec} hold with  $k= \frac{(t+1)^2+1}{t}\le t+3$. Thus, 
there exists $s\le t+2$  and a set of vertices $U\subset V$, $|U|\le s-1$, 
such that $G-U$ is a union of $s$ vertex-disjoint, 2-connected graphs, $G'_i=(V'_i, E'_i)$.
Note that since $|U|\le t+1$ and  $\ell > 4K(t+1)$ are large, 
\begin{equation}\label{eq:deltaup}
\delta(G'_i) \ge \delta(G)-|U| > \frac{2(2\ell-1)}{3} - (t+1) > \ell + K(t+1).
\end{equation}
Moreover, clearly,  $|V'_i|\le |V| <(t+1)2\ell$, so Lemma \ref{l:small} applied to $G'_i$, with $c=t+1$, gives
\begin{equation*}\label{eq:Gi}
|V'_i| \le 2\ell - 1 \quad \textrm{for}\quad i=1,2,\dots,s.
\end{equation*}

Now, for every $i=1,2,\dots,s$, we define
\[
U_i = \{u\in U: \deg_G(u, V'_i) \ge 4t\},\quad  V_i = V'_i\cup U_i, \quad \textrm{and}\quad\ G_i = G[V_i].
\]
We will show that the sets $V_1, V_2, \dots, V_s$ 
satisfy the conditions \ref{it:1}-\ref{it:4} of the hypothesis of Lemma~\ref{l:V}. 

In order to verify \ref{it:1} observe that since the minimum degree of $G$ is large, i.e.\ $\delta(G)\ge 8t^2$, every vertex $u\in U$ belongs to at least one of the sets $U_i$, and therefore $V = V_1 \cup V_2 \cup \dots\cup V_s$. 

To prove that $|V_i|\le 2\ell -1$, let us assume that $|V_i|\ge 2\ell$. Now  take any subset     $\hat U_i$ of  $U_i$, with 
$|\hat U_i| = 2\ell - |V_i'|$ elements and set $H_i=G[V'_i\cup \hat U_i]$. Note that  $H_i$ has $2\ell$ vertices. 
We will argue that $H_i$ is hamiltonian. 
To this end, consider the closure of $H_i$. 
From \eqref{eq:deltaup} we know that all vertices from $V'_i$ have degree at least $\delta(G'_i)> \ell+K(t+1)$, 
so in the closure of $H_i$ the set $V'_i$ spans a clique of size at least $2\ell-|U|\ge 2\ell -t-1$. On the other hand, each vertex from $\hat U_i$ has in $V'_i$ at least $4t$ neighbours, so the closure of $H_i$ is the complete graph and therefore, by Lemma~\ref{l:BC}, $H_i$ is hamiltonian.
However it means that $C_{2\ell}\subseteq H_i\subseteq G$  which contradicts our assumption that $G$ is $C_{2\ell}$-free. Consequently, 
for every $i=1,2,\dots, s$, we have $|V_i|\le 2\ell-1$, as required by  \ref{it:2}. 

Note  that from \eqref{eq:deltaup} it follows that    $|V'_i|> \delta(G'_i)> \ell$. Since $U\setminus U_i$ sends at most $4t|U|\le 4t(t+1)<\ell$ edges to the set $V'_i$, there exists a vertex $v_i\in V'_i\subseteq V_i$ which has all its neighbours in $G_i$. It means however that, since $\xoverline G\not\supseteq K_{1,n}$, 
the set $V\setminus V_i$, which contains only vertices which are not adjacent to $v_i$,  has at most $n-1$ elements, and so \ref{it:3} holds.

Finally, to verify \ref{it:4} consider an auxiliary bipartite graph $F=(V_F, E_F)$, where $V_F=\{V'_1, V'_2, \dots, V'_s\}\cup U$ and 
	\[
	E_F=\{uV'_i:u\in U_i\}.
	\]
	We claim that $F$ is a forest. Indeed, assume for the sake of contradiction that $F$ contains a cycle $C=V'_{i_1}u_{j_1}\dots V'_{i_w}u_{j_w}V'_{i_{w+1}}$, ${i_1} = {i_{w+1}}$. Observe that every vertex $u_{j_x}$, $x= 1,2,\dots,w$, has at least  two neighbours in both sets $V'_{i_x}$ and $V'_{i_{x+1}}$.  Moreover,  $\delta(G'_i) > \ell+1$ and $|V'_i| \le 2\ell - 1$, so from Lemma~\ref{l:W} it follows that any two vertices of $V'_i$ can be connected by a path of length $y$ for every $y=2,3,\dots, |V_i|-1$.  Therefore, since  $w\le |U|\le t+1\le \ell/4$, the existence of $C$ in $F$ 
	implies the existence of  a cycle $C_{2\ell}$ in $G$, contradicting the fact that $G$ is  $C_{2\ell}$-free.
	
	Since  $F$ is a forest it contains at most $|U|+s-1$ edges, i.e.
	\[
\sum_{u\in U}\deg_F(u) \le |U|+s-1.
	\]
	Note that in the sum $|V_1|+|V_2|+\cdots+|V_s|$ each vertex from $\bigcup_i V'_i = V\setminus U$ is counted once, and 
	each vertex  $u\in U$ is counted precisely $\deg_F(u)$ times, so
	\[
	|V_1|+\dots +|V_s| = |V| -|U|+\sum_{u\in U}\deg_F(u)  \le  |V|+s-1,
	\]
	as required by \ref{it:4}.

Now we can apply Lemma~\ref{l:V} and infer that $|V|\le f_t(\ell,n)$ while we have assumed that $|V|=f_t(\ell,n)+1$. This final contradiction completes the
proof of the upper bound for $R(C_{2\ell}, K_{1,n})$ and, together with \eqref{eq:lower}, concludes the proof of Theorem~\ref{thm1}.
\end{proof}


\begin{bibdiv}
	\begin{biblist}
		
\bib{B}{article}{
	author={Bondy, J. A.},
	title={Pancyclic graphs. I},
	journal={J. Combinatorial Theory Ser. B},
	volume={11},
	date={1971},
	pages={80--84},
}	

\bib{BC}{article}{
	author={Bondy, J. A.},
	author={Chv\'{a}tal, V.},
	title={A method in graph theory},
	journal={Discrete Math.},
	volume={15},
	date={1976},
	number={2},
	pages={111--135},
}	
		
\bib{BS}{article}{
	author={Bondy, J. A.},
	author={Simonovits, M.},
	title={Cycles of even length in graphs},
	journal={J. Combinatorial Theory Ser. B},
	volume={16},
	date={1974},
	pages={97--105},
}		
		
\bib{BJ}{article}{
	author={Bukh, B.},
	author={Jiang, Z.},
	title={A bound on the number of edges in graphs without an even cycle},
	journal={Combin. Probab. Comput.},
	volume={26},
	date={2017},
	number={1},
	pages={1--15},
}		

		\bib{D}{article}{
	author={Dirac, G. A.},
	title={Some theorems on abstract graphs},
	journal={Proc. London Math. Soc. (3)},
	volume={2},
	date={1952},
	pages={69--81},
}

\bib{EG}{article}{
	author={Erd\H{o}s, P.},
	author={Gallai, T.},
	title={On maximal paths and circuits of graphs},
	journal={Acta Math. Acad. Sci. Hungar.},
	volume={10},
	date={1959},
	pages={337--356 (unbound insert)},
}


\bib{ESi}{article}{
	author={Erd\H{o}s, P.},
	author={Simonovits, M.},
	title={A limit theorem in graph theory},
	journal={Studia Sci. Math. Hungar},
	volume={1},
	date={1966},
	pages={51--57},
}

\bib{ESt}{article}{
	author={Erd\"{o}s, P.},
	author={Stone, A. H.},
	title={On the structure of linear graphs},
	journal={Bull. Amer. Math. Soc.},
	volume={52},
	date={1946},
	pages={1087--1091},
}

	\bib{FS}{article}{
		author={F\"{u}redi, Z.},
		author={Simonovits, M.},
		title={The history of degenerate (bipartite) extremal graph problems},
		conference={
			title={Erd\H{o}s centennial},
		},
		book={
			series={Bolyai Soc. Math. Stud.},
			volume={25},
			publisher={J\'{a}nos Bolyai Math. Soc., Budapest},
		},
		date={2013},
		pages={169--264},
	}
		
\bib{GHS}{article}{
	author={Gould, R. J.},
	author={Haxell, P. E.},
	author={Scott, A. D.},
	title={A note on cycle lengths in graphs},
	journal={Graphs Combin.},
	volume={18},
	date={2002},
	number={3},
	pages={491--498},
}

\bib{LPS}{article}{
	author={Lubotzky, A.},
	author={Phillips, R.},
	author={Sarnak, P.},
	title={Ramanujan graphs},
	journal={Combinatorica},
	volume={8},
	date={1988},
	number={3},
	pages={261--277},
}

\bib{W}{article}{
	author={Williamson, J. E.},
	title={Panconnected graphs. II},
	journal={Period. Math. Hungar.},
	volume={8},
	date={1977},
	number={2},
	pages={105--116},
}

\bib{VZ}{article}{
	author={Voss, H.-J.},
	author={Zuluaga, C.},
	title={Maximale gerade und ungerade Kreise in Graphen. I},
	language={German},
	journal={Wiss. Z. Tech. Hochsch. Ilmenau},
	volume={23},
	date={1977},
	number={4},
	pages={57--70},
}

\bib{ZBC}{article}{
	author={Zhang, Y.},
	author={Broersma, H.},
	author={Chen, Y.},
	title={Narrowing down the gap on cycle-star Ramsey numbers},
	journal={J. Comb.},
	volume={7},
	date={2016},
	number={2-3},
	pages={481--493},
}
		
\end{biblist}
\end{bibdiv}
\end{document}